\documentclass[reqno,final,a4paper]{amsart}
\usepackage{color}
\usepackage{amsmath, amssymb, amsthm, thmtools}
\usepackage[colorlinks=true,allcolors=blue
]{hyperref}
\usepackage{mathrsfs}
\usepackage{mathtools}
\usepackage[linesnumbered,ruled,vlined]{algorithm2e}

\SetCommentSty{mycommfont}
\SetKwInOut{Input}{input}\SetKwInOut{Output}{output}

\usepackage[noadjust]{cite}
\usepackage[noabbrev,capitalize,nameinlink]{cleveref}
\crefname{equation}{}{}
\usepackage{fullpage}
\usepackage{graphics}
\usepackage{tikzscale}
\usepackage{pifont}
\usepackage{tikz}
\usetikzlibrary{fit, backgrounds, shapes, calc}
\usepackage{pgfplots}
\pgfplotsset{compat=1.12}
\usepackage{bbm}
\usepackage[T1]{fontenc}
\usetikzlibrary{arrows.meta}

\usepackage{environ}
\usepackage{framed}
\usepackage{url}

\usepackage[noend]{algpseudocode}
\usepackage[labelfont=bf]{caption}
\usepackage{framed}
\usepackage[framemethod=tikz]{mdframed}
\usepackage{appendix}
\usepackage{graphicx}
\usepackage[textsize=tiny]{todonotes}
\usepackage{tcolorbox}
\usepackage{xcolor}
\usepackage[shortlabels]{enumitem}
\crefformat{enumi}{#2#1#3}
\crefrangeformat{enumi}{#3#1#4 to~#5#2#6}
\crefmultiformat{enumi}{#2#1#3}
{ and~#2#1#3}{, #2#1#3}{ and~#2#1#3}
\allowdisplaybreaks

\usepackage{ifthen}
\numberwithin{equation}{section}
\newtheorem{theorem}{Theorem}[section]

\newtheorem{lemma}[theorem]{Lemma}
\newtheorem{claim}[theorem]{Claim}
\crefname{claim}{Claim}{Claims}

\newtheorem*{question*}{Question}

\theoremstyle{definition}

\newtheorem{problem}[theorem]{Problem}

\newtheorem*{definition*}{Definition}

\newtheorem*{fact*}{Fact}

\crefname{fact}{Fact}{Facts}

\theoremstyle{remark}

\renewcommand{\Pr}{\mathbb{P}}

\newcommand{\eps}{\varepsilon}

\def\epsilon{\varepsilon}

\newenvironment{proofclaim}[1][Proof of claim]{\begin{proof}[#1]}{\end{proof}}

\let\originalleft\left
\let\originalright\right
\renewcommand{\left}{\mathopen{}\mathclose\bgroup\originalleft}
\renewcommand{\right}{\aftergroup\egroup\originalright}

\title{On the number of distinct spanning trees in pseudorandom graphs}

\author{Yiting Wang}\address{Institute of Science and Technology Austria,~Klosterneuburg,~3400,~Austria.}\email{yiting.wang@ist.ac.at}

\thanks{Yiting Wang is supported by the European Research Council (ERC), via grant agreements ``RANDSTRUCT'' No.\ 101076777.}

\date{\today}

\begin{document}

\begin{abstract}
A celebrated result of Otter says the number of distinct unlabelled spanning trees in $K_n$ is $\alpha^n$ up to subexponential factors for an absolute constant $\alpha>0$.
In this note, we prove that for every $0<\varepsilon<\alpha$, there are constants $C$ and $d_0$ such that every $(n,d,\lambda)$-graph with $d\geq d_0$ and $d/\lambda \geq C$ has at least $(\alpha-\eps)^n$ distinct unlabelled spanning trees.
\end{abstract}
\maketitle

\section{Introduction}
How many labelled spanning trees are there in a graph $G$? For a complete graph on $n$ vertices $K_n$, the answer is given by the classical Cayley's formula~\cite{Cayley}. Let $T(G)$ denote the set of labelled spanning trees in $G$. Cayley's formula states that $|T(K_n)| = n^{n-2}$. What about the number of unlabelled spanning trees? 
We say two unlabelled spanning trees $T$ and $T'$ are \emph{isomorphic}, denoted as $T\cong T'$, if there exists a bijection $\sigma:V(T)\rightarrow V(T')$ such that $\{u,v\}\in E(T)$ if and only if $\{\sigma(u), \sigma(v)\}\in E(T')$. 
Two unlabelled spanning trees are \emph{distinct} if they are not isomorphic.
Let $\mathcal T^u(G)$ denote the set of (distinct) unlabelled spanning trees in $G$. Because the size of the automorphism group can vary vastly for different unlabelled spanning trees, this is a much more difficult question. Remarkably, Otter~\cite{Otter} proved the following sharp estimate on $|T^u(K_n)|$:
\begin{theorem}[Otter~\cite{Otter}]\label{thm:Otter}
    There exist absolute constants $C \approx 0.535$ and $\alpha \approx 2.956$ such that 
    \[
        |\mathcal T^u(K_n)| = (1+o(1))Cn^{-5/2} \alpha^n\,.
    \]
\end{theorem}
The constant $\alpha$ is often referred to as \emph{Otter's tree constant}.

Recently, Lee~\cite{Lee} initiated the study of estimating $|T^u(G)|$ for connected $d$-regular graphs. He obtained a universal lower bound for this graph family:
\begin{theorem}[{Lee~\cite[Theorem 1.4]{Lee}}]\label{thm:Lee}
    There exists a universal constant $d_0\in \mathbb N$ such that the following holds: Let $G$ be a connected $d$-regular graph on $n$ vertices with $d \geq d_0$, then $|T^u(G)|\geq e^{n/2000}$.
\end{theorem}
Resolving a conjecture by Lee, Bitonti, Michel and Scott~\cite{Bitonti-Michel-Scott} obtained a polynomial lower bound on $|T^u(G)|$ for connected graphs with minimum degree $d$.

In this paper, we are interested in a different direction: Instead of proving a universal lower bound, what additional conditions on $G$ (other than being $d$-regular) guarantee that $|T^u(G)|$ is almost as large as $|T^u(K_n)|$? 
The pseudorandom graph model we work with is the class of spectral expanders, i.e.\ $(n,d,\lambda)$-graphs. These are $d$-regular graphs on $n$ vertices whose adjacency matrix eigenvalues $d = \lambda_1\geq \lambda_2\geq \cdots \geq \lambda_n$ satisfy $\max\{\lambda_2, |\lambda_n|\}\leq \lambda$. For a detailed introduction on $(n,d,\lambda)$-graphs, refer to the survey by Krivelevich and Sudakov~\cite{survey}.

Here is our main result:
\begin{theorem}\label{thm:main}
    Let $\alpha\approx 2.956$ be Otter's tree constant.
    For every $0<\eps <\alpha$, there exist constants $d_0 = d_0(\eps)$, $C= C(\eps)$ such that the following holds: Let $G$ be an $(n,d,\lambda)$-graph with $d\geq d_0$ and $d/\lambda \geq C$, then $|T^u(G)|\geq (\alpha -\eps)^n$.
\end{theorem}
Note that some spectral-gap assumption is necessary. Indeed, a disconnected $d$-regular graph has $\lambda = d$ and hence $d/\lambda = 1$ but has no spanning tree.
Also, since $G$ is a subgraph of $K_n$, the lower bound is sharp up to the $\eps$ in the exponential base.

\section{Proof}
Throughout the paper, $n$ is an asymptotic parameter and we always assume $n$ is sufficiently large. We write $\mathcal T_n$ for the family of trees on $n$ vertices.
Define 
\[
    \mathcal T(\Delta,\rho,n) =\{T\in \mathcal T_n: \Delta(T)\leq \Delta, T \text{ has at least $\rho n$ leaves}\}\,.
\]
The proof consists of two steps: First, we show that for any $0<\eps<\alpha$, there exist $\Delta$ and $\rho$ such that $|\mathcal T(\Delta, \rho,n)|\geq (\alpha -\eps)^n$. Then, we invoke the following tree universality result for pseudorandom graphs:
\begin{theorem}[{Pavez-Signé~\cite[Theorem 4.2]{Pavez-Signe}}]\label{thm:ndlambda-universal} For all $\Delta \in \mathbb N$, there exist positive constants $C$ and $K$ such that the following holds for all sufficiently large $d\in \mathbb N$. If $G$ is an $(n,d,\lambda)$-graph with $d/\lambda\geq C$, then $G$ contains a copy of every tree $T\in \mathcal T(\Delta, K\lambda/d, n)$. 
\end{theorem}

For technical reasons, we will also need Otter's corresponding estimate for rooted unlabelled trees, also known as Pólya trees. Given two Pólya trees $T$ and $T'$, rooted at $r$ and $r'$ respectively, we say they are isomorphic if there exists a bijection $\sigma: V(T)\rightarrow V(T')$ such that $\sigma(r) = r'$ and that $\{u,v\}\in E(T)$ if and only if $\{\sigma(u),\sigma(v)\}\in E(T')$. Two Pólya trees are distinct if they are not isomorphic. 

\begin{theorem}[Otter~\cite{Otter}]\label{thm:Otter-unlabelled}
    There exists absolute constant $C' \approx 0.44$ such that the number of distinct Pólya trees in $K_n$ is $(1+o(1))C'n^{-3/2} \alpha^n$, where $\alpha\approx 2.956$ is Otter's tree constant.
\end{theorem}

Here is the enumeration result in the first step:
\begin{lemma}\label{lem:enumeration}
    Fix $0<\epsilon<\alpha$. There exists $\Delta = \Delta(\eps)$, $\rho = \rho(\eps)$ such that  $|\mathcal T(\Delta, \rho, n)|\geq (\alpha-\eps)^n$, where $\alpha$ is Otter's tree constant.
\end{lemma}
\begin{proof}
    We will explicitly construct a family of unlabelled (non-rooted)  trees of the desired size.
     Let $a_k$ denote the number of unlabelled rooted trees on $k$ vertices. From~\Cref{thm:Otter-unlabelled}, we know that $\lim_{k\rightarrow \infty}(a_k)^{1/k} = \alpha$. Let $K\geq 2$ be such that $(a_K)^{1/K}> \alpha -\eps/3$. Define $L = \lfloor n/K\rfloor-6$ and $R = n - LK$. Note that $6K\leq R < 7K$.
    Let $T_1,\dots, T_L$ be arbitrary unlabelled rooted trees on $K$ vertices with roots $x_1,\dots, x_L$ respectively.
    Given $T_1,\dots, T_L$, we construct an unlabelled (non-rooted) tree on $n$ vertices in three steps:
    \begin{enumerate}
        \item Place trees: Place $T_1,\dots, T_L$ on pairwise disjoint vertex sets.
        \item Build the spine: Add edges $\{x_1,x_2\},\dots, \{x_{L-1},x_L\}$.
        \item Add pendant paths: Add a path of length $a = 2K+1$ to $x_1$ and a path of length $b= R - a$ to $x_L$. These two paths are vertex disjoint and the vertices on those two paths, other than $x_1$ and $x_L$, are vertex disjoint from $(\cup_{i=1}^L V(T_i))$. 
    \end{enumerate}
    Let $T$ denote the resulting tree. Note that $T$ is on $a + b + LK = n$ vertices. Below is an illustration of the procedure.
    \begin{figure}[h]
        \centering
        \includegraphics[width=0.7\linewidth]{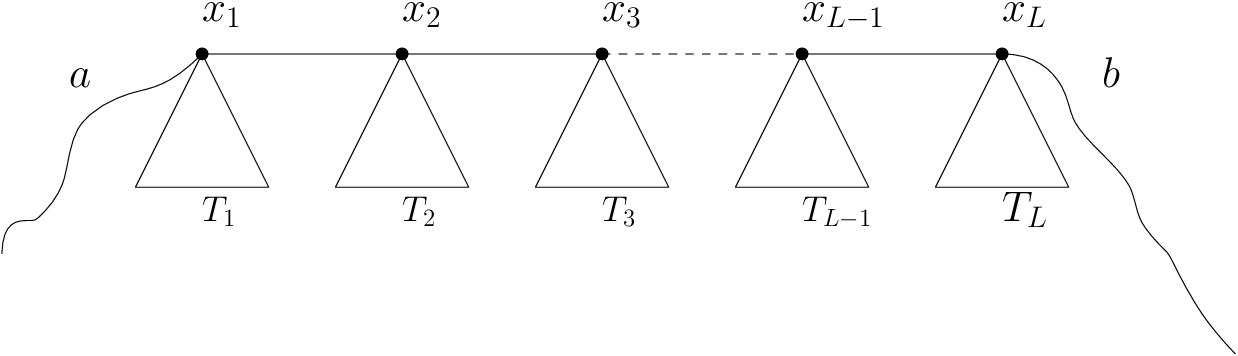}
        \caption{An example tree $T$}
        \label{fig:placeholder}
    \end{figure}
    
    \begin{claim}
        The map $(T_1,\dots, T_L)\mapsto T$ is injective.
    \end{claim}
    \begin{proofclaim}
        First, since $T_i$ has only $K$ vertices, the only maximal pendant paths of length greater than $K$ in $T$ are the two paths added in Step 3. Hence these two paths and their attachment vertices $\{x_1,x_L\}$ are identifiable from $T$.
        Moreover, since $b = R-a \geq 4K-1> a$, the two paths have distinct lengths. From there, we can distinguish $x_1$ from $x_L$. Lastly, the spine is the unique path connecting $x_1$ to $x_L$ in $T$ and so the ordered sequence $x_1,x_2,\dots, x_L$ is identifiable. Consequently, every distinct choice of $(T_1,\dots, T_L)$ gives rise to a distinct tree $T$, which proves the map is injective.
    \end{proofclaim}
    Fix an arbitrary tree $T$ obtained this way. It is easy to see the maximum degree is at most $2 + (K-1) = K+1$. 
    Also, each rooted tree $T_i$ has a non-root leaf, and this vertex remains a leaf in the final tree. Hence the final tree has at least $L = \lfloor n/K\rfloor - 6 \geq n/2K$ leaves. Thus, taking $\Delta = K+1$ and $\rho = 1/(3K)$ is enough. It only remains to check how many $(T_1,\dots, T_L)$ there are. 
    By the choice of $K$, there are at least 
    \[
       (a_K)^L = \bigl(a_K^{1/K}\bigr)^{KL} >(\alpha - \eps/3)^{n - 7K} > (\alpha -\eps)^n 
    \]
    options for $(T_1,\dots, T_L)$, as desired.
\end{proof}

Now we prove~\Cref{thm:main}.
\begin{proof}[Proof of~\Cref{thm:main}]
    Let $\Delta, \rho$ be given by~\Cref{lem:enumeration} and $\mathcal T = \mathcal T(\Delta, \rho, n)$. Note that $|\mathcal T|\geq (\alpha-\eps)^n$. Applying~\Cref{thm:ndlambda-universal}, there exist $K_0,C_0,d_0 >0$ such that a $(n,d,\lambda)$-graph with $d\geq d_0$ and $d/\lambda \geq C_0$ contains a copy of every tree $T\in \mathcal T(\Delta, K_0\lambda/d, n)$. Let $C = \max\{C_0, K_0/\rho\}$. Since $K_0\lambda/d\leq K_0/C\leq \rho$, we know $\mathcal T\subseteq \mathcal T(\Delta, K_0\lambda /d,n)$. Thus,
    we conclude that an $(n,d,\lambda)$-graph with $d/\lambda \geq C$ contains a copy of every $T\in \mathcal T$, and therefore has at least $|\mathcal T|\geq (\alpha-\eps)^n$ distinct unlabelled spanning trees.
\end{proof}

\section{Concluding remarks}
In this paper, we established that for every $0<\eps<\alpha$, there are constants $C$ and $d_0$ such that for every $(n,d,\lambda)$-graph $G$ with $d\geq d_0$ and $d/\lambda\geq C$ has at least $(\alpha-\eps)^n$ distinct unlabelled spanning trees.
The proof crucially uses a spanning trees universality result with many leaves. If one only enforces the bounded degree condition and not the many leaves condition, then currently the best known result is due to Hyde, Morrison, M\"{u}yesser and Pavez-Sign\'{e}~\cite{Hyde-et-al}, who showed $d/\lambda =\Omega_\Delta(\log ^3n)$ guarantees containment of all bounded degree spanning trees. 
Note that this would only give~\Cref{thm:main} under the stronger assumption $d/\lambda = \Omega_\Delta(\log^3n)$.
An old conjecture of Alon, Krivelevich and Sudakov~\cite{Alon-et-al} says that $d/\lambda = \Omega_\Delta(1)$ is sufficient for bounded degree spanning tree universality, which is an intriguing open question.

Another interesting direction to explore is to prove the optimal anti-concentration inequality for spanning trees in $(n,d,\lambda)$-graphs. The question appears in the concluding remarks of~\cite{Lee}, which we reiterate here:
\begin{problem}\label{prob}
    Fix $\eps > 0$.
    Let $G$ be an $(n,d,\lambda)$-graph and $T$ be an arbitrary spanning tree in $G$. Under what conditions on $G$ is it true that $\Pr(T\cong T') \leq e^{-(1-\eps)n}$ where $T'$ is a uniformly random spanning tree in $G$? 
\end{problem}
Results of this type directly imply there are many distinct unlabelled spanning trees, but not the other way around. In particular,~\Cref{thm:main} does not shed light on~\Cref{prob}.

\end{document}